\documentclass[12pt]{article}
\usepackage{amsthm,amsfonts,hyperref}
\usepackage{amsmath}

\newtheorem{thm}{Theorem}
\newtheorem{cor}{Corollary}
\newtheorem{lem}{Lemma}

\newtheorem{rem}{Remark}
\newtheorem{dfn}{Definition}

\newcommand{\C}{\mathbb{C}}

\title{Spectral Lower Bounds for the Quantum Chromatic Number of a Graph \\ Part II}

\author{
Pawel Wocjan\thanks{\texttt{Pawel.Wocjan@ibm.com},
IBM Quantum, IBM T.J. Watson Research Center, Yorktown Heights, NY 10598, USA} \quad
Clive Elphick\thanks{\texttt{clive.elphick@gmail.com}, School of Mathematics, University of Birmingham, Birmingham, UK} \quad 
Parisa Darbari\thanks{\texttt{pdarbar@Knights.ucf.edu}, Department of Computer Science, University of Central Florida, USA}
}

\begin{document}

\maketitle

\begin{abstract}
Hoffman proved that a graph $G$ with eigenvalues $\mu_1 \ge \ldots \ge \mu_n$ and chromatic number $\chi(G)$ satisfies:
\[
\chi \ge 1 + \kappa
\] 
where $\kappa$ is the smallest integer such that 
\[
\mu_1 + \sum_{i=1}^{\kappa} \mu_{n+1-i} \le 0.
\]
We strengthen this well known result by proving that $\chi(G)$ can be replaced by the quantum chromatic number, $\chi_q(G)$, where for all graphs $\chi_q(G) \le \chi(G)$ and for some graphs $\chi_q(G)$ is significantly smaller than $\chi(G)$. We also prove a similar result, and investigate implications of these inequalities for the quantum chromatic number of various classes of graphs, which improves many known results.
For example, we demonstrate that the Kneser graph $KG_{p,2}$ has $\chi_q = \chi = p - 2$.
\end{abstract}

\section{Introduction}

For any undirected and connected graph $G$ let $V$ denote the set of vertices where $|V| = n$, $A$ denote the adjacency matrix,  $\mu_{\max} = \mu_1 \ge \mu_2 \ge \ldots \ge \mu_n$ denote the eigenvalues of $A$. The multiplicity of the maximum eigenvalue $\mu_{\max}$ is $1$ since $G$ is connected.\footnote{The Perron-Frobenius theorem implies that the maximum eigenvalue $\mu_{\max}$ has multiplicity $1$ because the adjacency matrix $A$ of the connected graph $G$ is irreducible.} Let $z$ denote the corresponding (unit) eigenvector.

Let $\chi(G)$ denote the chromatic number and $\omega(G)$ the clique number. Let $\chi_q(G)$ denote the quantum chromatic number, as defined by Cameron \emph{et al} \cite{cameron07}. It is known that $\chi_q(G) \le \chi(G)$, and Mancinska and Roberson \cite{mancinska162}  found a graph on 14 vertices with $\chi(G) > \chi_q(G)$, which they suspect is the smallest possible example. There exist graphs for which $\chi_q(G)$ is significantly smaller than $\chi(G)$. 

Elphick and Wocjan \cite{elphick19} proved that many spectral lower bounds for $\chi(G)$ are also lower bounds for $\chi_q(G)$, using linear algebra techniques called pinching and twirling. In this paper, we prove that stronger lower bounds on the chromatic number are also lower bound on the quantum chromatic number.

The following purely combinatorial definition of  the quantum chromatic number is due to \cite[Definition 1]{mancinska162}. For $d>0$, let $I_d$ and $0_d$ denote the identity and zero matrices in $\C^{d\times d}$. For $c>0$, let $[c]=\{1,\ldots,c\}$.
 
\begin{dfn}[Quantum $c$-coloring]\label{def:Q-coloring}
A quantum $c$-coloring of the graph $G=(V,E)$  is a collection of orthogonal projectors $\{ P_{v,k} : v\in V, k\in [c]\}$ in 
$\C^{d\times d}$ such that
\begin{itemize}
\item for all vertices $v\in V$
\begin{eqnarray}
\sum_{k\in[c]} P_{v,k} & = & I_d \quad\quad \mathrm{(completeness)} \label{eq:complete}
\end{eqnarray}
\item for all edges $vw\in E$ and for all $k\in[c]$
\begin{eqnarray}
P_{v,k} P_{w,k} & = & 0_d \quad\quad \mathrm{(orthogonality)} \label{eq:orthogonal}
\end{eqnarray}
\end{itemize}
The quantum chromatic number $\chi_q(G)$ is the smallest $c$ for which the graph $G$ admits a quantum $c$-coloring for some dimension $d>0$.
\end{dfn}

According to the above definition, any classical $c$-coloring can be viewed as a $1$-dimensional quantum coloring, where we set $P_{v,k}=1$ if vertex $v$ has color $k$ and $P_{v,k}=0$, otherwise. Therefore, quantum coloring is a relaxation of  classical coloring. As noted in \cite{mancinska162}, it is surprising that the quantum chromatic number can be strictly and even exponentially smaller than the chromatic number for certain families of graphs.

We use the following alternative characterization of the quantum chromatic number due to \cite{elphick19}. Before stating this characterization, we briefly review the definition of pinching. 

\begin{dfn}[Pinching]
    Let $P_k\in\C^{m\times m}$ for $k\in [c]$ be orthogonal projectors such that they form a resolution of the identity, that is, 
    \begin{equation}
        \sum_{k\in[c]} P_k = I_m.
    \end{equation}
    The operation $\mathcal{D} : \C^{m\times m} \rightarrow \C^{m\times m}$ defined by 
    \begin{equation}
        X \mapsto \mathcal{D}(X) = \sum_{k\in [c]} P_k X P_k
    \end{equation}
    is called pinching. We say that it annihilates $X$ if $\mathcal{D}(X)=0$. 
\end{dfn}

The following theorem is proved as Theorem~1 in \cite{elphick19}. For the sake of completeness, we include a condensed proof below. We use $\{e_v : v \in V\}$ to denote the standard basis in $\C^n$.

\begin{thm}\label{thm:quantum_pinching}
Let $\{ P_{v,k} : v\in V, k\in [c]\}$ be an arbitrary quantum $c$-coloring of $G$ in $\mathbb{C}^d$. Then, the following block-diagonal projectors
\begin{align}
    P_k &= \sum_{v\in V} e_v e_v^\dagger \otimes P_{v,k} \in \mathbb{C}^{n\times n} \otimes \mathbb{C}^{d\times d}    
\end{align}
define a pinching operation that annihilates $A\otimes I_d$, that is,
\begin{equation}
    \sum_{k\in[c]} P_k (A\otimes I_d) P_k = 0.
\end{equation}
\end{thm}
\begin{proof}
We have
\begin{align}
    \sum_{k\in [c]} P_k 
    &= 
    \sum_{k\in [c]} \sum_{v\in V} e_v e_v^\dagger \otimes P_{v,k} \\
    &=
    \sum_{v\in V} e_v e_v^\dagger \otimes \sum_{k\in [c]} P_{v,k} \\
    &=
    \sum_{v\in V} e_v e_v^\dagger \otimes I_d \\
    &= I_n \otimes I_d.
\end{align}
This shows that the orthogonal projectors $P_k$ form a resolution of the identity, that is, form a pinching operation.

For $v,w\in V$, let $a_{vw}$ denote the entries of the adjacency matrix. For $k\in[c]$, we have
\begin{equation}
    P_k(A\otimes I_d)P_k = \sum_{v,w\in V} a_{vw} \cdot e_v e_w^\dagger \otimes P_{v,k} P_{w,k}.
\end{equation}
Whenever $a_{vw}=1$, or equivalently $vw\in E$, the corresponding orthogonal projectors $P_{v,k}$ and $P_{w,k}$ must be orthogonal. This shows that the above sum is equal to $0_d$, that is, the corresponding pinching operation annihilates $A\otimes I_d$.
\end{proof}

\begin{rem}
The classical case corresponds simply to the special case $d=1$. In this case, the projectors $P_k$ are a diagonal in the standard basis of $\C^n$ and each projector corresponds to a color class, that is, each projector $P_k$ projects onto the subspace spanned by the standard basis vectors corresponding to the vertices that have been colored with the $k$th color.
\end{rem}


\section{New bounds for the quantum chromatic number}

We use $\uparrow$ to indicate that the eigenvalues are sorted in increasing order.  The $i$th smallest eigenvalue of a hermitian matrix $X\in\C^{n\times n}$ is denoted by $\mu_i^\uparrow(X)$ so that
\begin{align}
    \mu_1^\uparrow(X) \le \mu_2^\uparrow(X) \le \ldots \le \mu_n^\uparrow(X).
\end{align}
Similarly, we use $\downarrow$ to indicate that the eigenvalues are sorted in decreasing order.  The $i$th largest eigenvalue of $X$ is denoted by $\mu_i^\downarrow(X)$ so that
\begin{align}
    \mu_1^\downarrow(X) \ge \mu_2^\downarrow(X) \ge \ldots \ge \mu_n^\downarrow(X).
\end{align}
When the eigenvalues are sorted in decreasing order, we often omit $\downarrow$ so that
\begin{align}
    \mu_1(X) \ge \mu_2(X) \ge \ldots \ge \mu_n(X).
\end{align}
We have $\mu^\uparrow_i(X)=\mu_{n+1-i}(X)$.  It would be somewhat inconvenient to always have to write $n+i-1$, so this is why the $\uparrow$ notation simplifies the indices.

The following lemma is a standard interlacing result in matrix analysis. We describe it in detail since it is the main result that we rely on to prove the new stronger bounds on the quantum chromatic number.

\begin{lem}\label{lem:eig_eq1} Let $X\in\C^{n\times n}$ be a hermitian matrix and let $S\in\C^{n\times m}$ be a matrix such that $S^\dagger S=I_m$, that is, its $m$ column vectors $s_1,\ldots,s_m$ are orthonormal vectors. Then,
\begin{equation}
    \mu_i^\uparrow(X) \le \mu_i^\uparrow(S^\dagger X S) \\
\end{equation}
for $i\in[m]$.
\end{lem}
\begin{proof}
This follows from the Courant-Weyl-Fisher theorem stating
\begin{equation}
    \mu_i^\uparrow(X) = \min_{\mathcal{M}_i} \max_{x\in\mathcal{M}_i} x^\dagger X x,
\end{equation}
where the minimum is taken over subspaces $\mathcal{M}_i$ of dimension $i$ and the maximum is taken over unit vectors $x\in\mathcal{M}_i$. By multiplying $X$ by $S^\dagger$ and $S$ from the left and right, respectively, we effectively restrict the subspaces $\mathcal{M}_i$ to have
the form
\begin{equation}
    \mathcal{M}_i = \{ S y : y \in \mathcal{N}_i \}
\end{equation}
where $\mathcal{N}_i$ is a subspace of $\C^m$ of dimension $i\in [m]$.
\end{proof}

We obtain the following corollary by using the identity $\mu_i^\uparrow(-X)=-\mu_i(X)$ and applying the above lemma to the matrix $-X$.

\begin{cor}\label{cor:eig_eq2}
Let $X\in\C^{n\times n}$ and $S\in\C^{n\times m}$ with $S^\dagger S=I_m$ as in Lemma \ref{lem:eig_eq1}. Then
\begin{equation}
    \mu_i(X) \ge \mu_i(S^\dagger X S).
\end{equation}
\end{cor}

In the following, we will only use this corollary with $i=1$, that is,
\begin{equation}
    \mu_{\max}(X) \ge \mu_{\max}(S^\dagger X S). 
\end{equation}

\subsection{First new bound for \texorpdfstring{$\chi_q(G)$}{Lg}}

\begin{thm}[First bound on quantum chromatic number] \label{thm:firstbound}
Let $\chi_q(G)$ be the quantum chromatic number of a connected graph $G$ with adjacency matrix $A$.
Let $\kappa$ be the smallest integer such that
\begin{equation}
0 \ge \mu_{\max}(A) + \sum_{i=1}^\kappa \mu^\uparrow_{i}(A) 
\end{equation}
holds. Then the quantum chromatic number is bounded from below by
\begin{align}
    \chi_q(G) \ge 1 + \kappa. 
\end{align}
\end{thm}

\begin{proof}
Let $\{P_{v,k} : v\in V, k\in[c]\}$ be any quantum $c$-coloring in dimension $d$. Construct the corresponding collection $\{P_k : k\in[c]\}$ of block-diagonal projectors as in Theorem~\ref{thm:quantum_pinching}.

Let $z\in\C^n$ denote the unique eigenvector of $A$ corresponding to the largest eigenvalue $\mu_{\max}(A)$. Let $f_j\in\C^d$ for $j\in[d]$ denote the standard basis vectors. Let $s_1,\ldots,s_m$ be an orthonormal
basis of the subspace
\begin{align}\label{eq:span}
    \mathcal{S} = \mathrm{span} \{ P_k (z \otimes f_j) : k \in [c], j \in [d] \}.
\end{align} 
Its dimension $m$ satisfies
\begin{align}\label{Ineq:m-bounds}
    d \le m \le c d.
\end{align}
For the lower bound, observe that the $d$ orthogonal vectors $z\otimes f_j$ are contained in $\mathcal{S}$ since 
\begin{align}
    z\otimes f_j = (I_n \otimes I_d) \Big( z \otimes f_j \Big) = 
    \big( \sum_{k\in [c]} P_k \big) (z \otimes f_j) = \sum_{k\in [c]} P_k (z \otimes f_j).
\end{align}
For the upper bound, observe that there are exactly $cd$  vectors in eq.~(\ref{eq:span}) and, thus, $m$ cannot be larger than $cd$.

Let $S\in\C^{nd\times m}$ be the matrix with $s_1,\ldots,s_m$ as column vectors. The following two arguments show that the largest eigenvalue of the matrix $S^\dagger(A\otimes I_d)S$ is equal to $\mu_{\max}(A)$ and its multiplicity is equal to $d$. First, there exist $d$ orthogonal vectors $y_1,\ldots, y_d\in\C^m$ such that $S y_j = z\otimes f_j$ since the latter vectors are contained in the subspace $\mathcal{S}$, or equivalently, the column space of $S$. We have
\begin{align}
    S^\dagger(A\otimes I_d) S y_j &= S^\dagger (A\otimes I_d) (z\otimes f_j) \\
    &= \mu_{\max}(A) S^\dagger (z\otimes f_j) \\
    &= \mu_{\max}(A) S^\dagger S y_j \\
    &= \mu_{\max}(A) y_j.
\end{align}
Second, using Corollary \ref{cor:eig_eq2}, the largest eigenvalue of $S^\dagger(A\otimes I_d)S$ cannot be greater than the largest eigenvalue of $A\otimes I_d$.

We can always choose the orthonormal basis vectors $s_1,\ldots,s_m$ such that for each $i\in [m]$ there exists a unique $k_i\in [c]$ with
\begin{equation}\label{eq:unique}
    P_{k_i} s_i = s_i \mbox{ and } P_k s_i = 0 \mbox{ for all } k\neq k_i.
\end{equation}
This is because $\mathcal{S}=\bigoplus_{k\in[c]} \mathcal{S}_k$, where $\mathcal{S}_k = \mathrm{span}\{P_k(z\otimes f_j) : j \in [d]\}$ since the projectors $P_k$ form a resolution of the identity.  

We now see that the diagonal entries of the matrix $S^\dagger(A\otimes I)S$ must 
all be zero since
\begin{equation}
    (S^\dagger(A\otimes I)S)_{ii} = s_i^\dagger (A\otimes I_d) s_i = s_i^\dagger P_{k_i} (A \otimes I_d) P_{k_i} s_i = 0.
\end{equation}
For the last equality we  used that $P_k(A\otimes I_d)P_k=0$ for all $k\in[c]$. 

So using Lemma~\ref{lem:eig_eq1}, we obtain
\begin{align*}
    0 
    &= \mathrm{tr}(S^\dagger(A\otimes I)S) \\
    &= \sum_{i=1}^m \mu_i^{\uparrow}(S^\dagger(A\otimes I) S) \\
    &= \sum_{i=1}^{m-d} \mu_i^{\uparrow}(S^\dagger(A\otimes I) S)  + d \cdot \mu_{\max}(A) \\
    &\ge \sum_{i=1}^{m-d} \mu_i^{\uparrow}(A\otimes I) + d \cdot \mu_{\max}(A).
\end{align*}
Now let $\kappa_d$ be the smallest integer such that
\begin{equation}
    0 \ge \sum_{i=1}^{\kappa_d} \mu_i^{\uparrow}(A\otimes I) + d \cdot \mu_{\max}(A).
\end{equation}
Using (\ref{Ineq:m-bounds}), we have
\begin{equation}
    (c - 1) d = cd - d \ge m-d \ge \kappa_d.
\end{equation}
Note that $\kappa_d \ge (\kappa - 1)d + 1$ must hold because otherwise the condition that $\kappa = \kappa_1$ is minimal would be violated.
This implies $c - 1 \ge \lceil \kappa_d / d \rceil = \lceil \kappa - 1 + 1/d \rceil = \kappa$.
In particular, this hold for a quantum $c$-coloring attaining $\chi_q(G)$ so that $\chi_q(G)\ge 1 + \kappa$.
\end{proof}

A weaker version of Theorem~\ref{thm:firstbound}, with $\chi(G)$ replacing $\chi_q(G)$, was proved by Hoffman \cite{hoffman70} in 1970. This theorem immediately implies that 
\begin{equation}
1 + \frac{\mu_1}{|\mu_n|} \le \chi_q(G) \le \chi(G),
\end{equation}
which was proved in \cite{elphick19} using different techniques.

The proof of the following bound generalises a proof due to Haemers (\cite{haemers79}, \cite{haemers95}) from the classical to the quantum chromatic number.
\subsection{Second new bound for \texorpdfstring{$\chi_q(G)$}{Lg}}

\begin{thm}[Second bound on quantum chromatic number]\label{thm:secondbound}
For any connected graph $G$ with $\mu_2 >0$:
\begin{equation}
\chi_q(G) \ge 1 + \min \left\{ g, \frac{|\mu_n(A)|}{\mu_2(A)} \right\},
\end{equation}
where $g$ is the multiplicity of $\mu_n(A)=\mu_{\min}(A)$.
\end{thm}

\begin{proof}
Consider an arbitrary quantum $c$-coloring in dimension $d$. Assume that $c \le g$.

Let $\mathcal{S}$ be defined as in the proof of the previous theorem in (\ref{eq:span}). Let $\mathcal{T}$ be the subspace spanned by the eigenvectors corresponding to the $cd$ smallest eigenvalues $\mu^{\uparrow}_1(A\otimes I),\ldots,\mu^{\uparrow}_{cd}(A\otimes I)$. We now show that there
exists a non-zero unit vector $y$ with
\begin{equation}
    y \in \mathcal{S}^\perp \cap \mathcal{T}.
\end{equation}
To this end, define $\mathcal{R} = \mathrm{span}\{ z \otimes f_j : j \in [d] \}$. Observe that both $\mathcal{S}^\perp$ and $\mathcal{T}$ are contained in the subspace $\mathcal{R}^\perp$ and 
\begin{equation}
    \dim \mathcal{R}^\perp = 
    nd - d < nd = 
    (nd - cd) + cd \le 
    \dim \mathcal{T} + \dim \mathcal{S}^\perp.
\end{equation}
For $i\in [c]$, define $y_i = P_i y$. Let $m$ be the number of $y_i$ that are non-zero. We now show that at least two of them (w.l.o.g.~$y_1$ and $y_2$) must be non-zero, that is, $m\ge 2$.  First of all, at least one must be non-zero because otherwise we would have $0\neq y = (I_n\otimes I_d) y = \big(\sum_{k\in[c]} P_k\big) y = \sum_{k\in[c]} y_k = 0$. Now assume that only $y_1$ were non-zero, or equivalently, $y=P_1 y$. But this leads to the contradiction
\begin{equation}
    0 > \mu^\uparrow_{c}(A) = \mu^\uparrow_{cd}(A\otimes I_d) \ge y^\dagger(A \otimes I_d) y = y^\dagger P_1(A\otimes I_d) P_1 y = 0,
\end{equation}
where the first inequality holds because $c\le g$, the second inequality holds because $y\in \mathcal{T}$, and the last equality holds because $P_1 (A\otimes I_d) P_1=0$. The latter follows from the fact that $P_1$ is one of the projectors forming a pinching that annihilates $A\otimes I_d$.

Define the orthonormal vectors $s_i = y_i / \|y_i\|$ for $i\in [m]$ and $S$ to be the matrix whose columns are $s_i$. Define the matrix $X=A\otimes I_d - \Delta \cdot z z^\dagger \otimes I_d$, where $\Delta = \mu_{\max} - \mu_{\min}$ and $z$ is the (unit) eigenvector corresponding to $\mu_{\max}$. 

Since $y$ is in the column space of $S$, the smallest eigenvalue of $S^\dagger X S$ is at most $y^\dagger X y$, which in turn is at most $\mu^\uparrow_{cd}(A\otimes I_d) = \mu^{\uparrow}_c(A)=\mu_n(A)$ as $y\in\mathcal{T}$ and $c\le g$. Also, it holds that $\mu_{\max}(S^\dagger X S)\le\mu_{\max}(X)=\mu_2(A)$.

We now show that the trace of $S^\dagger X S$ is equal to $0$. The diagonal entries of $S^\dagger X S$ are all zero because 
\begin{align}
    (S^\dagger X S)_{ii} 
    &= 
    s_i^\dagger X s_i \, \propto \, y_i^\dagger X y_i \\
    &=
    y^\dagger P_i \big(A\otimes I_d - \Delta \cdot z z^\dagger \otimes I_d\big) P_i y \\
    &= 
    y^\dagger P_i (A\otimes I_d) P_i y - \Delta \cdot \sum_{j\in [d]} y^\dagger P_i (z z^\dagger \otimes f_j f_j^\dagger) P_i y \\
    &= 0
\end{align}
where we used that $P_i(A\otimes I_d) P_i=0$ and $y\perp P_i(z\otimes f_j)$ for each $j$. The latter holds as $y\in\mathcal{S}^\perp$ and $P_i(z\otimes f_j)\in\mathcal{S}$.

Combining that $S^\dagger X S$ is traceless with the above bounds on its minimum and maximum eigenvalues yields that
\begin{equation}
0 = \mathrm{tr}(S^\dagger X S) = \sum_{i\in[m]} \mu^\uparrow_i(S^\dagger X S) \le \mu_n(A) + (m-1) \mu_2(A),
\end{equation}
which completes the proof.
\end{proof}

A weaker version of this bound,  with $\chi(G)$ replacing $\chi_q(G)$, is already known, for example in Corollary 3.6.4 in \cite{brouwer10}.

We note that both Theorems are also valid for weighted adjacency matrices of the form $W \circ A$, where $W$ is an arbitrary Hermitian matrix and $\circ$ denotes the Hadamard product (also called the Schur product).  An example of using a weighted adjacency matrix is to replace $A$ with the normalized adjacency matrix $\mathcal{A} = D^{-1/2} A D^{-1/2}$, where $D$ is the diagonal matrix of vertex degrees, in both bounds for $\chi_q(G)$. This choice of weight matrix reproduces the lower bound for $\chi(G)$ in Theorem 2.2  in Coutinho \emph{et al} \cite{coutinho19}, once account is taken of the differences in notation.

\section{Implications for quantum chromatic number}

\subsection{Strongly regular graphs (SRGs)}

Elphick and Wocjan \cite{elphick19} discussed implications of their results for the quantum chromatic number. For example they demonstrated using an inertial bound that the Clebsch graph has $\chi_q(G) = 4$. Since the Clebsch graph has spectrum $(5^1, 1^{10}, -3^5)$ this also follows immediately from Theorem \ref{thm:secondbound}. The generalised quadrangle, $GQ(2,4)$, on 27 vertices has $\chi = 6$, but in \cite{elphick19} the authors were only able to show that $\chi_q \ge 5$. The spectrum of $GQ(2,4)$ is $(10^1, 1 ^{20}, -5^6)$, so from Theorem \ref{thm:secondbound} it follows that $\chi_q = 6$. 

Both of these graphs are strongly regular, and Theorems \ref{thm:firstbound} and 3 can be used to calculate  the quantum chromatic number of many strongly regular graphs (SRGs). For example the Kneser graph $K_{p,2}$ (with  $p \ge 4$) has $\chi = p - 2$ and spectrum $((p-2)(p-3)/2^1, 1^{p(p-3)/2}, (3 - p)^{p-1})$, which using Theorem \ref{thm:secondbound} implies $\chi_q = \chi = p - 2$. The Hoffman-Singleton graph, SRG(50, 7, 0, 1) has spectrum$(7^1, 2^{28}, -3^{21})$ and $\chi = 4$. Theorem \ref{thm:firstbound} implies $\chi_q = 4$ also.

Fiala and Haemers \cite{fiala06} identified (see their Theorem 10.1) all SRGs with $\chi = 5$. So, using Theorem \ref{thm:firstbound}, SRG$(15, 8, 4, 4$) and  SRG($25, 8, 3, 2$) have $\chi_q = 5$; and using Theorem \ref{thm:secondbound}, SRG($21, 10, 3, 6$) and SRG($25, 16, 9, 12$) have $\chi_q = 5$. The Gewirtz graph, SRG($56, 10, 0 2$), has $\chi = 4$ and spectrum $(10^1, 2^{35}, -4^{20})$; so using Theorem \ref{thm:firstbound} it has $\chi_q = 4$. 

The Higman-Sims graph is SRG($100, 22, 0, 6$). Its spectrum is equal to $(22^1, 2^{77}, -8^{22})$ and it has $\chi = 6$ (see \cite{fiala06}). Theorem \ref{thm:firstbound} implies $\chi_q \ge 4$ and Theorem \ref{thm:secondbound} implies $\chi_q \ge 5$. We do not however know whether $\chi_q = 5$ or $6$. Similarly the $M_{22}$ graph is SRG($77, 16, 0 4$) with spectrum $(16^1, 2^{55}, -6^{21})$ has $\chi = 5$ (see \cite{fiala06}). Theorems \ref{thm:firstbound} and \ref{thm:secondbound} imply $\chi_q \ge 4$, but we do now know whether $\chi_q = 4$ or $5$.

\subsection{Non-SRGs}\label{sec:orth-graph}

The orthogonality graph, $\Omega(n)$, has vertex set the set of $\pm1-$vectors of length $n$, with two vertices adjacent if they are orthogonal. With $4|n$ (see \cite{mancinska16}), it is known that $\chi_q(\Omega(n)) = n$ but $\chi(\Omega(n))$ is exponential in $n$.  A proof that $\chi_q(\Omega(n)) = n$ is as follows. It is immediate from the definition of $\Omega(n)$ that $\xi'(\Omega(n)) = n$, and it is known that $\chi_q(G) \le \xi'(G)$, where $\xi'(G)$ is the normalized orthogonal rank of $G$ \cite{wocjan19}. However, using Theorem \ref{thm:firstbound}  and results in section 4.3 of \cite{godsil06} on the eigenvalues of orthogonality graphs we have that:

\[
\chi_q(\Omega(n)) \ge 1 + \frac{\mu_1}{|\mu_n|} = 1 + \frac{1\cdot 3 \cdots (n-3) \cdot (n-1)}{1 \cdot 3 \cdots (n - 3)} = n.
\]

\begin{thm}\label{thm:dimension}
The orthogonality graph $\Omega(n)$ has a quantum coloring in dimensions $pn$, where $p$ is a positive integer.
\end{thm}

\begin{proof}
We can construct a quantum coloring of $\Omega(n)$ using $n$ colors as follows.  Let dimension $d = n$, let $U = \mathrm{diag}(1, \omega, \ldots, \omega^{n-1})$ be a unitary matrix where $\omega = e^{2\pi i/n}$, and let $z_v$ denote the $\pm1$ vector of length $n$ assigned to vertex $v$. Then let

\[
P_{v,k} = U^k z_v z_v^\dagger (U^\dagger)^k: v\in V, k \in [c].
\]

It is straightforward that this collection of orthogonal projectors satisfy the completeness and orthogonality conditions in Definition \ref{def:Q-coloring}, so this completes the quantum coloring with $d = n$.

For $d=pn$, where $p>1$, let 
  \[
  \widetilde{P}_{v,k}=P_{v,k}\otimes I_p: v\in V, k\in [c].
  \]
  This new collection of orthogonal projectors also satisfy the completeness and orthogonality conditions in Definition \ref{def:Q-coloring}.
\end{proof}
We note in passing that with $4|n$, a proof that $\omega(\Omega(n)) = n$, where $\omega(G)$ denotes the clique number of $G$, would provide a proof of the Hadamard Conjecture, which dates from 1867.

Vertex Transitive($12, 27$) and Vertex Transitive($12, 54$) are examples of non-SRGs for which Theorem \ref{thm:secondbound} is exact with $\chi_q = 4$. Barbell graphs and irregular complete $q-$partite graphs have $\chi_q = \chi$, using Theorem \ref{thm:firstbound}.

\subsection{Hoffman colorings}

Any graph for which
\[
1 + \frac{\mu_1}{|\mu_n|} = \chi(G),
\]
is said to have a Hoffman coloring. All such graphs therefore have $\chi_q(G) = \chi(G)$. Examples include SRG($49, 12, 5, 2$) which has $\chi_q = 7$ and the Schlafli graph SRG($27, 16, 10, 8$) which has $\chi_q = 9$. Haemers and Touchev investigated graphs with Hoffman colorings, and Table 1 in \cite{haemers96} lists many such SRGs with up to 100 vertices.

\section{Open questions}

The pentagon $(C_5)$ demonstrates that both of the bounds in this paper are not lower bounds for the vector chromatic number or the fractional chromatic number. The orthogonal rank, $\xi(G)$, is incomparable to $\chi_q(G)$.  We do not know whether $\xi(G)$ can replace $\chi_q(G)$ in Theorems \ref{thm:firstbound} and \ref{thm:secondbound}.

We have shown that the Kneser graph $K_{p,2}$ has $\chi_q = \chi$, but is this true for all Kneser graphs? Are there any strongly regular graphs with $\chi_q < \chi?$

In Definition \ref{def:Q-coloring} the dimension $d$ is any finite positive integer. Let $\chi_d(G)$ denote the smallest $c$ for which graph $G$ admits a quantum $c$-coloring in dimension $d$. From  Theorem~\ref{thm:dimension}, we know that $\chi_1(\Omega(n)) = \chi(\Omega(n))$ which is exponential in $n$, but $\chi_{pn}(\Omega(n)) = n$ for $p$ a positive integer. This raises the question of what is the value of $\chi_d(\Omega(n))$ for $d \not = pn?$ In particular does $\chi_{n+1}(\Omega(n)) = n?$

 \subsection*{Acknowledgement}
 
 We would like to thank David Roberson for insightful comments on an earlier version of this paper.

\end{document}